\documentclass[11pt,reqno]{amsart}
\usepackage[a4paper, margin=1in]{geometry}
\usepackage{amsmath, amsthm, amssymb}
\usepackage{graphicx, array}
\usepackage{longtable}
\usepackage{epstopdf}
\usepackage{xcolor}
\usepackage{tikz-cd}
\usepackage{cite}
\usetikzlibrary{arrows.meta}
\usepackage{wrapfig}
\usetikzlibrary{calc}
\usetikzlibrary{backgrounds}
\usetikzlibrary{automata}
\usetikzlibrary{positioning}
\usepackage{scalerel}
\usepackage{enumitem}

\makeatletter %this allows item labels to be referenced, note it doesn't work with hyperref
\def\namedlabel#1#2{\begingroup
    #2%
    \def\@currentlabel{#2}%
    \label{#1}\endgroup
}
\makeatother

\newtheorem{Thm}{Theorem}[section]
\newtheorem{Prop}[Thm]{Proposition}
\newtheorem{Lem}[Thm]{Lemma}
\newtheorem{Cor}[Thm]{Corollary}

\newtheorem*{Thm*}{Theorem}

{\theoremstyle{definition}
\newtheorem{Def}[Thm]{Definition}}
{\theoremstyle{definition}
}
{\theoremstyle{definition}
}

\newenvironment{Proof}{\rm \trivlist\item[\hskip \labelsep{\bf
Proof.\quad}]}{\hfill\qed\par\medskip\endtrivlist}

\newcommand{\id}{\mathrm{id}}

\newcommand{\ol}{\overline}

\renewcommand{\epsilon}{\varepsilon}
\newcommand{\Inv}{\operatorname{Inv}}
\newcommand{\FInv}{\operatorname{FInv}}
\newcommand{\FIM}{\operatorname{FIM}}

\newcommand{\sub}{\operatorname{Sub}}
\newcommand{\csub}{\operatorname{CSub}}

\newcommand{\red}{\operatorname{red}}

\newcommand{\rrel}{\mathrel{\mathcal{R}}}

\newcommand{\sigmarel}{\mathrel{\sigma}}

\newcommand{\m}{\mathfrak m}

\numberwithin{equation}{section}

\title[$E$-unitary and $F$-inverse monoids, and closure operators]{$E$-unitary and $F$-inverse monoids, and closure operators on group Cayley graphs}

\author[N\'{o}ra Szak\'{a}cs]{N\'{o}ra Szak\'{a}cs}
\address{School of Mathematics, University of Manchester, Manchester M13 9PL, UK}
\email{nora.szakacs@manchester.ac.uk}

\begin{document}

\maketitle

\begin{center}
\textit{Dedicated to Mária B. Szendrei on the occasion of her 70th birthday.}
\end{center}

\begin{abstract}
We show that the category of $X$-generated $E$-unitary inverse monoids with greatest group image $G$ is equivalent to the category of $G$-invariant, finitary closure operators on the set of connected subgraphs of the Cayley graph of $G$. Analogously, we study $F$-inverse monoids in the extended signature $(\cdot, 1, ^{-1}, ^\m)$, and show that the category of $X$-generated $F$-inverse monoids  with greatest group image $G$ is equivalent to the category of $G$-invariant, finitary closure operators on the set of all subgraphs of the Cayley graph of $G$. As an application, we show that presentations of $F$-inverse monoids in the extended signature can be studied by tools analogous to Stephen's procedure in inverse monoids, in particular, we introduce the notions of $F$-Schützenberger graphs and $P$-expansions.
\end{abstract}

\section{Introduction}

In inverse semigroup theory, the study of inverse semigroups that can be understood in terms of their greatest group images and their semilattices of idempotents has played an important role. In particular, $E$-unitary inverse semigroups form the most widely studied subclass, one of the many equivalent definitions being that the greatest group morphism $\sigma^\sharp \colon S \to S/\sigma$ is idempotent pure, that is, the preimage of $1$ is the set of idempotents in $S$. From a universal algebraic point of view, $E$-unitary inverse semigroups form a quasivariety, which is precisely the Maltsev product of the varieties of semilattices and groups relative to the variety of inverse semigroups.

The structure of $E$-unitary inverse semigroups is completely understood due to a classical result of McAlister \cite{McAl}, which represents them as semigroups associated to triples $(\mathcal X, \mathcal Y, G)$, where $G$ is a group (the greatest group image of the semigroup), $\mathcal Y$ is a semilattice (the semilattice of its idempotents) and $\mathcal X$ is a poset containing $\mathcal Y$, acted upon by the group $G$. The poset $\mathcal X$ is hardest to grasp in terms of the semigroup, and McAlister's theorem has been reproved many times providing various ways to conceptualize it.
In \cite{BS03}, Steinberg interprets $\mathcal X$ as a set of connected subgraphs of the Cayley graph $\Gamma$ of the group $G$. 
We notice that we can rephrase this interpretation to view these subgraphs as fixed points of a closure operator,
expanding an idea that first appears in \cite{SzSz} in the special case of `finite-above' $E$-unitary inverse monoids. This leads us to show in Theorem \ref{Thm:catequivalent} that the category of all $X$-generated, $E$-unitary inverse monoids with greatest group image $G$ is equivalent to the category of all $G$-invariant, finitary closure operators on the poset of connected subgraphs of the group's Cayley graph (taken with respect to the generating set $X$). The $E$-unitary inverse monoid constructed from a given closure operator (Definition \ref{Lem:closure-operator-yields-S}) generalizes the construction of the Margolis-Meakin expansion of a group \cite{MM}, which arises from taking the identical closure operator, and is the initial object of the category. When the group $G$ is free on $X$, the Margolis-Meakin expansion of $G$ is the free inverse monoid on $X$, given in the form described famously by Munn. 

Steinberg's proof of the McAlister theorem is underpinned by the work of Stephen \cite{S90} on presentations of inverse monoids. Given an inverse monoid $S=\Inv\langle X \mid u_i=v_i,\ i \in I \rangle$, Stephen describes a procedure, now called Stephen's procedure, to iteratively build the strong components of the Cayley graph of $S$ from finite pieces by sewing on paths and applying Stallings' foldings. These strong components, called Sch\"utzenberger graphs, are in turn key to solving the word problem in $S$, which makes Stephen's procedure the fundamental way to study inverse monoids given by a presentation. 

The natural morphism $\sigma^\sharp \colon S \to S/\sigma$ defines a graph morphism from the Cayley graph of $S$ to the Cayley graph of $S/\sigma$. It is well-known that $S$ is $E$-unitary if and only if this graph morphism is injective on each of the strong components (i.e. the Sch\"utzenberger graphs). In Steinberg's interpretation, the set $\mathcal X$ consists exactly of the embedded copies of the Sch\"utzenberger graphs. Viewing Sch\"utzenberger graphs as naturally embedded into the group Cayley graph also simplifies Stephen's procedure in the $E$-unitary case: the graphs can be built directly in the group Cayley graph, which alleviates the need for Stallings' foldings. This version of Stephen's procedure is essentially what defines the closure operator in our interpretation.

$F$-inverse monoids are those inverse semigroups where each $\sigma$-class has a greatest element with respect to the natural partial order -- these are automatically unital, with $1$ being the greatest element of the $\sigma$-class containing the idempotents. $F$-inverse monoids are automatically $E$-unitary too, but have been studied widely in their own right: see \cite{AKSz} and the references therein.

Kinyon observed in a 2018 conference talk that $F$-inverse monoids form a variety (in the sense of universal algebra) in the signature $(\cdot, 1, ^{-1}, ^\m)$, where $s^\m$ is the greatest element in the $\sigma$-class of $s$. The authors in \cite{AKSz} describe the free object in the variety, and more generally, the initial object in the category of all $X$-generated $F$-inverse monoids with a greatest group image $G$. Their construction is a clever modification of the Margolis-Meakin expansion. Inspired by (and building on) this work, we adapt much of what is described above from $E$-unitary to $F$-inverse monoids.
In particular, we associate closure operators to $F$-inverse monoids (Definition \ref{Lem:S-yields-closure-F-inv}), which gives rise to a suitable analogue of Sch\"utzenberger graphs we call $F$-Sch\"utzenberger graphs (Definition \ref{def:F-Sch-graph}), and show that these play the same role in the word problem for $F$-inverse monoids as Sch\"utzenberger graphs do for inverse monoids. We also describe an analogue of Stephen's procedure to study presentations of $F$-inverse monoids. Analogously to our results on $E$-unitary inverse monoids, we prove that the category of all $X$-generated, $F$-inverse monoids with greatest group image $G$ is equivalent to the category of all finitary closure operators on the poset of (not necessarily connected) subgraphs of the group's Cayley graph. The identical closure operator, again, corresponds to the initial object of the categories, and the respective construction of $F$-inverse monoids is precisely the one given in \cite{AKSz}. As an application, we show that free inverse monoids are not finitely presented as $F$-inverse monoids.

\subsection*{Acknowledgements.} {The author would like to thank Mark Kambites and Benjamin Steinberg for helpful conversations.}

\section{Preliminaries}

We refer the reader to \cite{Law} on the basics on inverse semigroups, and to \cite{BS} on the basics of universal algebra. In this paper, all algebraic structures we consider will come with a chosen generating set. We say that an algebraic structure $S$ is \emph{$X$-generated via the map} $\iota \colon X \to S$ if $S$ is generated by $\iota(X)$. Note that we do not require $\iota$ to be an embedding.
Given two $X$-generated algebraic structures $S_1, S_2$ via the maps  $\iota_i \colon X \to S_i$, $i=1,2$, we say that a morphism $\psi\colon S_1 \to S_2$ respects the generators, or is \emph{canonical}, if $\iota_2=\psi \circ \iota_1$. 

All inverse semigroups we consider will be assumed to have an identity element, to avoid worrying about the technical difference between a monoid being $X$-generated as a semigroup vs being $X$-generated as a monoid. Our results can be adapted to inverse semigroups via the usual construction: any inverse semigroup $S$ can be turned into an inverse monoid $S^1$ by adjoining an external identity. 

We view inverse monoids as algebraic structures in the signature $(\cdot, 1, ^{-1})$, where they form a variety. Any element of an $X$-generated inverse monoid can be represented by an element of the so-called \emph{free monoid with involution} $(X \cup X^{-1})^*$, where $X^{-1}=\{x^{-1}: x \in X \}$ is a formal set of symbols in bijection with $X$. Given a word $w \in (X \cup X^{-1})^*$ and an $X$-generated inverse monoid $S$, we denote by $w_S$ the element of $S$ represented by $w$.

Given an $X$-generated inverse monoid (in particular, a group) $S$, the Cayley graph $\Gamma_X$ of $S$ is an edge-labeled digraph on the vertex set $V(\Gamma_X)=S$, 
where for each $s \in S$ and $x \in X \cup X^{-1}$, there is an
edge labeled by $x$ pointing from $s$ to $sx_S$. The label of an edge $e$ is denoted by $l(e)$, and the initial and terminal vertices of $e$ are denoted by $\alpha(e)$ and $\omega(e)$ respectively. If $S$ is a group, 
then $\Gamma_X$ is strongly connected, in particular every edge $e$ has an inverse pair $e^{-1}$ with $\alpha(e^{-1})=\omega(e),\ \omega(e^{-1})=\alpha(e)$ and $l(e^{-1})=l(e)^{-1}$. Furthermore, the left Cayley representation of the group induces a left action on its Cayley graph by edge-labeled automorphisms.

If $S$ is not a group, $\Gamma_X$ is never strongly connected, but when $e$ is an edge such that $\alpha(e)$ and $\omega(e)$ are in the same strong component of $\Gamma_X$, then $e$ does have an inverse pair $e^{-1}$ with the above properties (see for example \cite{S90}). The strong components of $\Gamma_X$ are called the \emph{Schützenberger graphs of $S$}, and the Schützenberger graph containing the vertex $s \in S$ is denoted by $S\Gamma(s)$. It follows by definition that the vertex set $V(S\Gamma(s))$ is exactly the $\rrel$-class of $s$. In particular, $V(S\Gamma(s))$ contains exactly one idempotent of $S$, namely $ss^{-1}$. If $w \in (X \cup X^{-1})^\ast$, then the Schützenberger graph $S\Gamma(w)$ of $w$ is defined as $S\Gamma(w_S)$.

In the paper, all graphs considered are digraphs edge-labeled by the set $X \cup X^{-1}$, in which the edge set is equipped with an involution $^{-1}$ such that $\alpha(e^{-1})=\omega(e),\ \omega(e^{-1})=\alpha(e)$ and $l(e^{-1})=l(e)^{-1}$. We call graphs with these properties \emph{$X$-graphs}. The $X$-graphs we consider (in particular, the examples above) are also \emph{deterministic}: for each vertex $v$ and each label $x \in X \cup X^{-1}$, there is at most one edge $e$ with $\alpha(e)=v$ and $l(e)=x$. By a subgraph of an $X$-graph, we mean a subset of edges and vertices closed under $\alpha$, $\omega$ and $^{-1}$ -- notice that this is also an $X$-graph. An $X$-graph is connected if and only if it is strongly connected, so we do not distinguish between the two notions. 

A path in an $X$-graph is either a non-empty sequence $e_1 e_2 \ldots e_n$ of consecutive edges (i.e. $\omega(e_i)=\alpha(e_{i+1})$), or a single vertex $v$, viewed as the empty path around that vertex $v$. The label of a path $e_1 \ldots e_n$ is defined to be the word $l(e_1) \cdots l(e_n) \in (X \cup X^{-1})^\ast$, and the label of a vertex $v$ is defined to be the empty word.
We compose connecting paths in the usual way, and define the inverse of a path $p=e_1  \ldots e_n$ to be the path $p^{-1}=e_n^{-1} \ldots e_1^{-1}$, and the inverse of an empty path as itself. We call two paths $p$ and $q$ \emph{coterminal} if $\alpha(p)=\alpha(q)$ and $\omega(p)=\omega(q)$.
 
Notice that in a deterministic $X$-graph, a path is uniquely determined by its initial vertex and label. In the Cayley graph of a group, paths starting from $1$ will play a special role, so we introduce the notation of $\ol w$ to denote the path labeled by $w$ starting at $1$. The path starting at $g$ labeled $w$ is then the image of the path $\ol w$ under the left action of $g$ on $\Gamma_X$, which we denote by $g \ol w$. The \emph{graph spanned by a path} $p$ is the subgraph induced by all edges and vertices occurring in the path $p$, and is denoted by $\langle p \rangle$.
 
It is proved in \cite{S90} that given an $X$-generated inverse monoid $S$ and $s \in S$, a word $w \in (X \cup X^{-1})^\ast$ labels a path in $S\Gamma(s)$ from $ss^{-1}$ to $s$ if and only if $w_S \geq s$. In particular, given two words $w_1, w_2 \in (X \cup X^{-1})^\ast$, we have $(w_1)_S=(w_2)_S$ if and only if $w_1$ labels a path in $S\Gamma(w_2)$ from $(w_2w_2^{-1})_S$ to $(w_2)_S$, and vice versa. It follows that if one can effectively construct arbitrarily large balls of the Sch\"utzenberger graphs $S\Gamma(w_S)$ for any $w \in (X \cup X^{-1})^\ast$, then one can decide the word problem in $S$. 

Stephen in \cite{S90} defines a procedure which, given an $X$-generated inverse monoid $S$ by a presentation $\Inv\langle X \mid u_i=v_i,\ i \in I\rangle$, and a word $w \in (X \cup X^{-1})^\ast$, constructs a sequence of finite graphs approximating (in a precise sense) the Sch\"utzenberger graph $S\Gamma(w_S)$. When $S$ is an $E$-unitary inverse monoid, this procedure simplifies a lot, so we restrict to describing this case. 

Let $S$ be an inverse monoid generated by $X$ via the map $\iota_S \colon X \to S$. Then $G=S/\sigma$ is also $X$-generated by $\iota_G= \sigma^\sharp \circ \iota_S$ where $\sigma^\sharp\colon S \to G$ is the natural projection onto $\sigma$-classes, moreover the map $\sigma^\sharp$ is by definition canonical. For any $s \in S$, this induces a graph morphism from $S\Gamma(s)$ to the Cayley graph $\Gamma_X$ of $G$ by mapping the vertex $t$ to $\sigma^\sharp(t)=t\sigma$. When $S$ is $E$-unitary, the map $\sigma^\sharp$ is injective on the $\rrel$-classes, so in particular this graph morphism is injective, and we can therefore view $S\Gamma(s)$ as a subgraph of $\Gamma_X$.

In the next claim, we sum up the relevant properties of Sch\"utzenberger graphs of $E$-unitary inverse monoids, based on \cite{S90}. While this is not explicitly stated in \cite{S90}, it easily follows from it and is well-known to experts.

\begin{Prop} \label{prop:stephen_main}
Let $S$ be an $X$-generated, $E$-unitary inverse monoid, and let $w \in (X \cup X^{-1})^*$. Let $G$ be the maximal group image of $S$, viewed as an $X$-generated group, and let $\Gamma_X$ be its Cayley graph. Then the following hold:
\begin{enumerate}
	\item\label{item:stephen_main_Sgraph} The Schützenberger graph $S\Gamma(w_S)$ is isomorphic to the smallest subgraph $\Delta$ of $\Gamma_X$ such that it contains the path $\ol w$, and 
	for any pair of vertices $g, h \in V(\Delta)$ and any pair of words $u, v \in (X \cup X^{-1})^*$ representing the same element of $S$, $\Delta$ contains a path from $g$ to $h$ labeled by $u$ if and only if it contains the path from $g$ to $h$ labeled by $v$.
	\item If $S$ is given by a presentation $\Inv\langle X \mid u_i=v_i,\ i \in I\rangle$, then it suffices to take $u,v$ to be a pair of equal relators in (\ref{item:stephen_main_Sgraph}).
	\item \label{item:stephen_main_language} A word $u \in (X \cup X^{-1})^\ast$ labels a path in $S\Gamma(w)$ from $1$ to $w_G$ if and only if $u_S \geq w_S$.
\end{enumerate}
\end{Prop}

It follows that given an $X$-generated inverse monoid $S$ and a word $w\in (X \cup X^{-1})^*$, $S\Gamma(w_S)$ can be approximated by the sequence of graphs $\Delta_1, \Delta_2, \ldots \subseteq \Gamma_X$ where $\Delta_1=\langle \ol w \rangle$, and $\Delta_{i+1}$ is obtained from $\Delta_i$ by performing what Stephen calls \emph{$P$-expansions}. 
More precisely, if $u, v\in (X \cup X^{-1})^*$ are such that $u_S=v_S$, and furthermore if $g, h \in V(\Delta_i)$ are vertices such that $\Delta_i$ contains a path from $g$ to $h$ labeled by $u$, but not by $v$, then the corresponding $P$-expansion adds the
subgraph $\langle g\ol v\rangle$ to $\Delta_i$. (Note that since $u_S=v_S$ and hence $u_G=v_G$, this path will end at $h$.) 
We define $\Delta_{i+1}$ to be the \emph{full $P$-expansion of $\Delta_i$}, defined as the union of $\Delta_i$ and the set of all such subgraphs $\langle g\ol v\rangle$, i.e. the graph obtained by performing all possible $P$-expansions on $\Delta_i$ simultaneously. Stephen's results imply that $S\Gamma(w_S)=\cup_{i=1}^\infty \Delta_i$. Furthermore, if $S$ is given by a presentation $\Inv\langle X \mid u_i=v_i,\ i \in I\rangle$ it suffices to perform full $P$-expansions with respect to pairs of equal relators.

\subsection{$F$-inverse monoids}

While $F$-inverse monoids are inverse monoids themselves, whenever we talk about $F$-inverse monoids, we consider them in the extended signature $(\cdot, 1, ^{-1}, ^\m)$, where $a^\m$ is the greatest element in the $\sigma$-class of $a$. 
$F$-inverse monoids form a variety in this signature, which was first observed by Kinyon, and is formally proved in \cite{AKSz}. There is a difference between an $F$-inverse monoid being $X$-generated as an $F$-inverse monoid (in the extended signature) and $X$-generated as an inverse monoid: the elements of $X$-generated $F$-inverse monoids can be represented by formal expressions $u_0 v_1^\m u_1  \cdots v_n^\m u_n$ where $u_i, v_i \in (X\cup X^{-1})^\ast$. Following \cite{AKSz}, we denote the set of all such terms by $\mathbb I\m_X$. It is useful to note that the identities $(w^{-1})^\m=(w^\m)^{-1}$ and $(u_0 v_1^\m u_1  \cdots v_n^\m u_n)^\m=(u_0 v_1 u_1  \cdots v_n u_n)^\m$ hold in the variety of $F$-inverse monoids -- both can be easily seen by definition, but a proof is also found in \cite{AKSz}. Groups are, naturally, $F$-inverse monoids with $^\m$ being the identical operation. As before, we use $w_S$ to denote the element of an $X$-generated $F$-inverse monoid $S$ represented by $w \in \mathbb I\m_X$.

Let $S$ be an $X$-generated $F$-inverse monoid with greatest group image $G$, which has Cayley graph $\Gamma_X$. It is a crucial idea in Proposition \ref{prop:stephen_main} that there is a correspondence between words in $(X \cup X^{-1})^*$ and paths in $\Gamma_X$ starting at a given vertex. For terms in $\mathbb I\m_X$, \cite{AKSz} introduces a corresponding notion of a \emph{journey}, consisting of paths interrupted by `jumps'. Formally, a journey is a sequence $j=(p_1, p_2, \ldots, p_n)$ of paths which are not connecting (i.e. $\omega(p_i)\neq \alpha(p_{i+1}))$, and we define $\alpha(j)=\alpha(p_1)$, and $\omega(j)=\omega(p_n)$. Like paths, connecting journeys can be composed: if $j_1=(p_1, p_2, \ldots, p_n), j_2=(q_1, q_2, \ldots, q_m)$ are journeys with $\omega(j_1)=\alpha(j_2)$, the journey $j_1j_2$ is defined to be $(p_1, p_2, \ldots, p_nq_1, q_2, \ldots, q_m)$. We similarly define $(p_1, p_2, \ldots, p_n)^{-1}$ as $(p_n^{-1}, p_{n-1}^{-1}, \ldots, p_1^{-1})$.

Given a term  $w=(u_0 v_1^\m u_1  \cdots v_n^\m u_n)$ in $\mathbb I\m_X$, the correponding journey $\ol w$ from the vertex $1$ consists of alternatingly traversing the paths labeled by $u_{i}$, followed by a `jumps' labeled by $v_{i+1}$, teleporting us from our current position $h \in V(\Gamma_X)$ to $h (v_{i+1})_G \in V(\Gamma_X)$. More precisely, $$\ol w=(\ol{u_0}, (u_0v_1)_G\ol{u_1}, \ldots, (u_0v_1 \ldots u_{n-1}v_n)_G \ol u_n).$$ We call $\ol w$ the journey labeled by $w$, starting at $1$. We similarly define the journey labeled by $w$, starting at $g$ to be the journey $g\ol w$.

Unlike for paths, given a journey $j=(p_1, p_2, \ldots, p_n)$ in $\Gamma_X$, there is no canonical way to associate a label in $\mathbb I\m_X$ to $j$: the `jumps' from $\omega(p_i)$ to $\alpha(p_{i+1})$ can be represented by any term $v_i^\m$ with $ \omega(p_i)(v_i)_G=\alpha(p_{i+1})$. However, the value $(v_i^\m)_S$ is unique: it is the greatest element in the $\sigma$-class $\omega(p_i)^{-1}\alpha(p_{i+1}).$ More generally, for any two terms $w_1, w_2 \in \mathbb I\m_X$ with $\ol w_1=\ol w_2$, we have $(w_1)_S=(w_2)_S$. Thus for any journey $j$, we define $l(j)_S$ to be the common $S$-value of any term $w \in \mathbb I\m_X$ labeling $j$. 
We define the subgraph of $\Gamma_X$ \emph{spanned by a journey} $j=(p_1, p_2, \ldots, p_n)$ to be the graph $\langle p_1 \rangle \cup \cdots \cup \langle p_n \rangle$. 

\subsection{Closure operators}

Given an $X$-graph $\Gamma$, denote the set of its subgraphs by $\sub \Gamma$, and the set of its connected subgraphs by $\csub \Gamma$. Both these sets are posets under inclusion, moreover, $\sub \Gamma$ forms a complete lattice with respect to set union and intersection. While $(\csub \Gamma, \subseteq)$ does not admit meets or joins, it is a \textit{relatively maximal lower bound complete} (in short, rmlb-complete) poset, which is defined as follows.  

Given a poset $(P, \leq)$ and a subset $Y \subseteq P$, let $Y^\downarrow$ denote the set of lower bounds for $Y$, i.e. the set of elements $x \in P$ with $x \leq y$ for all $y \in Y$. Let $\max Y^\downarrow$ denote its maximal elements, called the set of maximal lower bounds for $Y$. If $P$ happens to be a lattice, then $\max Y^\downarrow=\left\{ \bigwedge Y\right\}$.
We call a poset $P$ \textit{rmlb-complete} if for any $Y \subseteq P$ and $x \in P$, if $x$ is a lower bound for $Y$, then it is below a maximal lower bound for $Y$. 

The poset $(\csub \Gamma, \subseteq)$ is indeed rmbl-complete: given any set $Y$ of connected subgraphs, $\max Y^\downarrow$ is exactly the set of connected components of $\bigcap Y$, and any graph in $\csub \Gamma$ which is a subgraph of all graphs in $Y$ will be contained in one of these components. Notice that if $\Gamma$ is connected, then it is the maximum element of $\csub \Gamma$.

Recall that a \textit{closure operator} on a poset $(P, \leq)$ is a map $()^c \colon P \to P$ which is \textit{extensive}, that is, $x \leq x^c$, \textit{monotonous}, that is if $x \leq y$ then $x^c \leq y^c$, and \textit{idempotent}, that is $(x^c)^c =x^c$. The fixed points of $()^c$ are also called \textit{closed elements}.  We will denote the set of fixed points for $()^c$ by $Y_c$.
It is a well-known fact that if $P$ is a complete lattice, then $Y \subseteq P$ is the set of fixed points of some closure operator if and only if $Y$ is meet-closed and contains the maximum element of $P$. Given such a set $Y$, we can recover the corresponding closure operator as $x^c=\bigwedge \{y \in Y: x \leq y\}$.

This correspondence between closed elements and closure operators extends to rmlb-complete posets. By \cite[Theorem 4.4]{Ranz99}\footnote{There is a typo in the statement of \cite[Theorem 4.4]{Ranz99}: $\mathfrak M(Y)$ should be $M(Y)$.}, given an rmbl-complete poset $(P, \leq)$, a set $Y \subseteq P$ is the set of fixed points for a closure operator on $P$ if and only if for any $Z \subseteq Y$, $Y$ contains $\max Z^\downarrow$, and $Y$ contains the maximal elements of $P$. Given such a set $Y$, we can recover the corresponding closure operator as follows: $x^c$ is the (necessarily unique) element of $\max \{y \in Y: x\leq y\}^\downarrow$ above $x$. In particular, if the poset is $\csub \Gamma$,
then $Y \subseteq \csub \Gamma$ is the set of fixed points for some closure operator $()^c$ if for any $Z \subseteq Y$, $Y$ contains the connected components of $\bigcap Z$ as well as the connected components of $\Gamma$. Then given $\Delta \in \csub \Gamma$, we have that $\Delta^c$ is the unique component of $\bigcap \{\Xi \in Y \colon \Delta \subseteq \Xi\}$ containing $\Delta$.

For us, the key interest is the case of the Cayley graph $\Gamma_X$ of an $X$-generated group $G$.
We say that a closure operator on $(\sub \Gamma_X, \subseteq)$ or $(\csub \Gamma_X, \subseteq)$ is \textit{$G$-invariant} if the set of fixed points is invariant under the left action of $G$ on $\Gamma_X$, that is if for any (connected) subgraph $\Delta$, we have $(g\Delta)^c=g(\Delta^c)$. As usual, we call a closure operator on $(\sub \Gamma_X, \subseteq)$ or  $(\csub \Gamma_X, \subseteq)$ \textit{finitary} if $\Delta^c =\bigcup \{F^c: F \subseteq \Delta \hbox{ is finite}\}$ holds if for any (connected) subgraph $\Delta$, observe that this set union is indeed connected when $\Delta$ is connected and $()^c$ is a closure operator on $(\csub \Gamma_X, \subseteq)$. Finitary closure operators are completely determined by the closures of finite sets -- we will call these \textit{compact sets}.

\section{Closure operators and Sch\"utzenberger graphs}

\subsection{Closure operators associated to $E$-unitary and $F$-inverse monoids}

In this subsection, we show how $X$-generated $E$-unitary inverse monoids and $F$-inverse monoids give rise to closure operators on $\csub \Gamma_X$ and $\sub \Gamma_X$ respectively, where $\Gamma_X$ is the Cayley graph of their greatest group image. More generally, we associate closure operators to binary relations -- these will be relevant in Section \ref{Sec:presentations} where we study presentations of $F$-inverse monoids.

\begin{Def}[Closure operators from $E$-unitary inverse monoids]
	\label{def:E-unit-yields-closure}
	Given an $X$-generated group $G$ and a symmetric relation $R$ on $(X \cup X^{-1})^\ast$ where any pair of $R$-related words are equal in $G$, we define a closure operator on $\csub\Gamma_X$
	by its set of closed graphs: we define $\Delta \in \csub\Gamma_X$ to be closed whenever for any pair $(u,v) \in R$, and any pair of vertices $g,h \in V(\Delta)$, $u$ labels a path from $g$ to $h$ in $\Delta$ if and only if $v$ does. 
	
	In particular, any $X$-generated $E$-unitary inverse monoid $S$ with greatest group image $G$ gives rise to such a relation $R$, consisting of pairs $(u,v)$ with $u_S=v_S$ -- we denote the corresponding closure operator by $c_S$. 
\end{Def}

\begin{Lem}
	\label{Lem:S-yields-closure-E-unit}
	The operator $()^{c_R}$ is a $G$-invariant closure operator on $(\csub \Gamma_X, \subseteq)$. In particular, for any $X$-generated $E$-unitary inverse monoid $S$, the operator $()^{c_S}$ is a $G$-invariant closure operator on $(\csub \Gamma_X, \subseteq)$. 
\end{Lem}

\begin{proof}
	To see that $()^{c_R}$ is indeed a well-defined closure operator on $(\csub \Gamma_X, \subseteq)$, observe that for any set $Y$ of subgraphs defined as closed, the connected components of $\bigcap Y$ are also closed. Furthermore $\Gamma_X$ itself is closed, since any pair of words in $R$ label coterminal paths in $\Gamma_X$ from any vertex. $G$-invariance is also clear as the set of closed sets is invariant under the left action of $G$ by definition. 
	
	To see that $()^{c_R}$ is finitary, take $\Delta \in \csub\Gamma_X$ and consider the set $$\ol \Delta=\bigcup \{F^{c_R}: F \subseteq \Delta \hbox{ is finite}\},$$ we need $\ol \Delta=\Delta^{c_R}$. We first show that $\ol \Delta$ is closed: choose $(u, v) \in R$, then indeed if $\ol \Delta$ contains a path $p$ labeled by $u$, as $\langle p \rangle$ is finite, there are finite subsets $F_1, \ldots, F_n \subseteq \Delta$ with $\langle p \rangle \subseteq F_1^{c_R} \cup \cdots \cup F_n^{c_R}$. If $\langle p \rangle$ intersects each graph $F_i$, which we can assume without loss of generality, then $F_1 \cup \cdots \cup F_n$ is connected and 
	$$\langle p \rangle \subseteq F_1^{c_R} \cup \cdots \cup F_n^{c_R}\subseteq (F_1 \cup \cdots \cup F_n)^{c_R}.$$ But since $(F_1 \cup \cdots \cup F_n)^{c_R}$ is closed, it contains the coterminal path labeled by $v$. Hence $\ol \Delta$ is a closed, and as it contains $\Delta$, we have $\ol \Delta \supseteq \Delta^{c_R}$. Conversely, if $F \subseteq \Delta$, then $F^{c_R} \subseteq \Delta^{c_R}$, hence $\ol \Delta \subseteq \Delta^{c_R}$, proving equality.
	\end{proof}

The definition of $()^{c_S}$ in the $E$-unitary case is obviously inspired by Stephen's procedure. There is a natural analogue for $F$-inverse monoids: 

\begin{Def}[Closure operators from $F$-inverse monoids]
	\label{def:F-inverse-yields-closure}

	Given an $X$-generated group $G$ and a symmetric relation $R$ on $\mathbb I\m_X$ where any pair of $R$-related terms are equal in $G$, we define a closure operator on $\sub\Gamma_X$
	by its set of closed graphs: we define $\Delta \in \sub\Gamma_X$ to be closed whenever for any pair $(u,v) \in R$, and any pair of vertices $g,h \in V(\Delta)$, $u$ labels a journey from $g$ to $h$ in $\Delta$ if and only if $v$ does. 
	
	In particular, any $X$-generated $F$-inverse monoid $S$ with greatest group image $G$ gives rise to such a relation $R$, consisting of pairs $(u,v)$ with $u_S=v_S$ -- we denote the corresponding closure operator by $c_S$. 	
\end{Def}

We can state a lemma analogous to Lemma \ref{Lem:S-yields-closure-E-unit}, the proof is nearly verbatim the same with paths replaced by journeys, so we omit it.

\begin{Lem}
	\label{Lem:S-yields-closure-F-inv}
	The operator $()^{c_R}$ is a $G$-invariant closure operator on $(\sub \Gamma_X, \subseteq)$.
	In particular, for any $X$-generated $F$-inverse monoid $S$, the operator $()^{c_S}$ is a $G$-invariant closure operator on $(\sub \Gamma_X, \subseteq)$. 

\end{Lem}

Observe that for a closure operator $()^{c_S}$ associated to an $E$-unitary inverse monoid $S$, its compact sets are exactly the embedded copies of the Schützenberger graphs of $S$.
Indeed, any compact set is of the from $F^{c_S}$ with $F \in \csub \Gamma_X$ finite, and if $p$ is any path spanning $F$, then $F^{c_S}$ is isomorphic to $S\Gamma(l(p))$ by Proposition \ref{prop:stephen_main}.

This naturally gives rise to the idea of considering the compact sets in the $F$-inverse case to be $F$-inverse analogues of Schützenberger graphs.

\subsection{$F$-Sch\"utzenberger graphs of $F$-inverse monoids}
\label{subset:F-Sch-graph}

\begin{Def}[$F$-Schützenberger graphs of $F$-inverse monoids]
	\label{def:F-Sch-graph}
	Given an $X$-generated $F$-inverse monoid $S$, we define the $F$-Schützenberger graph $F\Gamma(w)$ of a term $w \in \mathbb I\m_X$ to be the subgraph $\langle \ol w \rangle^{c_S}$ of the Cayley graph of its greatest group image. For $s \in S$, we define the $F$-Schützenberger graph $F\Gamma(s)$ of $s$ to be $F\Gamma(w)$ for any word $w$ representing $s$.
\end{Def}

Note that $F\Gamma(s)$ is well-defined: if $w_S=u_S$ for some $w,u \in \mathbb I\m_X$, then $\langle \ol u \rangle  \subseteq \langle \ol w \rangle^{c_S}$ and $\langle \ol w \rangle  \subseteq \langle \ol u \rangle^{c_S}$ by Definition \ref{def:F-inverse-yields-closure} which yields $\langle \ol u \rangle^{c_S}  = \langle \ol w \rangle^{c_S}$.

We can define the obvious analogues of $P$-expansions and full $P$-expansions.

\begin{Def}[$P$-expansions for $F$-inverse monoids]
Let $S$ be an $X$-generated $F$-inverse monoid, and let $\Gamma_X$ be the Cayley graph of its greatest group image. Let $R$ be symmetric relation on $\mathbb I\m_X$ where any pair of $R$-related terms are $G$-equivalent. Given a subgraph $\Delta \in \csub \Gamma_X$, two $R$-related terms $u,v \in \mathbb I\m_X$, and a journey in $\Delta$ labeled by $u$ such that the coterminal journey in $\Gamma_X$ labeled by $v$ is not in $\Delta$, a \emph{$P$-expansion} on $\Delta$ is the operation of adding the subgraph spanned by this journey to $\Delta$.

The \emph{full $P$-expansion} of $\Delta$ is the graph obtained from $\Delta$ by performing all possible $P$-expansions on $\Delta$ simultaneously, i.e. for any journey $u$ with the properties above, we add the coterminal journey $v$.
\end{Def}

\begin{Prop}
\label{prop:FP-expansions}
For a term $w \in \mathbb I\m_X$, put $\Delta_0=\langle \ol w \rangle$, and for $i \geq 1$, let $\Delta_{i+1}$ be the full $P$-expansion of $\Delta_i$ with respect to the relation $R$. Then $\langle \ol w \rangle^{c_R}=\bigcup_{i=1}^\infty \Delta_i$. 
\end{Prop}

\begin{proof}
On the one hand, observe that $\bigcup_{i=1}^\infty \Delta_i$ is closed: if $u,v \in \mathbb I\m_X$ with $(u,v)\in R$, then any journey labeled by $u$, contained in $\bigcup_{i=1}^\infty \Delta_i$ is contained in some $\Delta_i$, and its `coterminal pair' labeled by $v$ will be contained in $\Delta_{i+1}$ by definition. Since $\langle \ol w \rangle \subseteq \bigcup_{i=1}^\infty \Delta_i$, this implies $\langle \ol w \rangle^{c_R} \subseteq \bigcup_{i=1}^\infty \Delta_i$. Conversely, $\Delta_i \subseteq \Delta_{i-1}^{c_R}$ for any $i$, and so $\Delta_i \subseteq \langle \ol w \rangle^{c_R}$ follows by an inductive argument, hence $\bigcup_{i=1}^\infty \Delta_i \subseteq \langle \ol w \rangle^{c_R}$. 
\end{proof}

We will explore how $P$-expansions relate to presentations on Section \ref{Sec:presentations}. For now, we prove the analogue of Proposition \ref{prop:stephen_main} (\ref{item:stephen_main_language}) for $F$-Schützenberger graphs. The following lemma is crucial, and is proved in greater generality in \cite[Lemma 4.6]{AKSz}.

\begin{Lem}\label{Lem:4.6}
If $u,v \in \mathbb I\m_X$ label coterminal journeys in $\Gamma_X$ such that $\langle u \rangle \subseteq \langle v \rangle$, then $u_S \geq v_S$.
\end{Lem}

\begin{Prop}
	\label{prop:language-of-F-Sch}
Let $S$ be an $X$-generated $F$-inverse monoid, and consider an $F$-Schützenberger graph $F\Gamma(w)$ for some $w \in \mathbb I\m_X$. Then $u \in \mathbb I\m_X$ labels a journey in $F\Gamma(w)$ from $1$ to $w_G$ if and only if $u_S \geq w_S$.
\end{Prop}

\begin{proof}
First note that if $u_S \geq w_S$, that is, if $(u w^{-1} w)_S=w_S$, then $F\Gamma(w)= \langle \ol w \rangle^{c_S}$ contains the journey of $\Gamma_X$ labeled by $u w^{-1} w$, starting at $1$. Since $u_S \geq w_S$ also implies $u_G=w_G$, the prefix of the journey above labeled by $u$ is a $1$ to $w_G$ journey, proving the right to left implication.

For the converse, by Proposition \ref{prop:FP-expansions}, if $u$ labels a journey in $F\Gamma(w)= \langle \ol w \rangle^{c_S}$, then $u$ is contained in a graph $\Delta$ obtained from $\langle \ol w \rangle$ by a finite number of elementary $P$-expansions with respect to the relation $R$ consisting of pairs representing equal elements in $S$. We prove by induction on the number of elementary $P$-expansions needed to construct $\Delta$ that if $u$ labels a $1$ to $w_G$ journey in $\Delta$, then $u_S \geq w_S$.

If $\Delta=\langle \ol w \rangle$, then the claim follows from Lemma \ref{Lem:4.6}. For the inductive step, assume $\Delta$ is constructed from $\Delta^\circ$ by an elementary $P$-expansion, and the claim holds for $\Delta^\circ$. Denote the journey added to $\Delta^\circ$ in the $P$-expansion by $p$, and the coterminal journey of $\Delta^\circ$ inducing the expansion by $q$, so $l(p)_S=l(q)_S$. Assume that $u$ labels a $1$ to $w_G$ journey in $\Delta$, we need to show $u_S \geq w_S$.
Let $j$ be a journey in $\Delta^\circ$ from $1$ to $w_G$ which spans $\Delta^\circ$ (such a journey can easily be constructed by exhausting all edges in an alternating sequence of edges and jumps).  Since $j$ traverses the vertex $\alpha(q)$, we can factor $j$ as $j_1j_2$ where $\omega(j_1)=\alpha(q)$. Then $j_1 qq^{-1} j_2$ is also a journey from $1$ to $w_G$ spanning $\Delta^\circ$, and by the inductive hypothesis, we have $l(j_1qq^{-1}j_2)_S \geq w_S$. But since $l(p)_S=l(q)_S$, we also have $l(j_1pp^{-1}j_2)_S=l(j_1qq^{-1}j_2)_S\geq w_S$, where $\langle j_1pp^{-1}j_2 \rangle =\Delta$, in particular $\langle \ol u \rangle \subseteq \langle j_1pp^{-1}j_2 \rangle$ and $\ol u$ and $j_1pp^{-1}j_2$ are both $1$ to $w_G$ journeys. Applying Lemma \ref{Lem:4.6} again, we have $u_S \geq  l(j_1pp^{-1}j_2)_S \geq w_S$.
\end{proof}

\section{The correspondence between closure operators, and $E$-unitary and $F$-inverse monoids}

In this section, we prove that for any $X$-generated group $G$ with Cayley graph by $\Gamma_X$, the category of $G$-invariant, finitary closure operators on $(\csub \Gamma_X, \subseteq)$, respectively on $(\sub \Gamma_X, \subseteq)$, is equivalent to the category of $X$-generated $E$-unitary, respectively $F$-inverse monoids with greatest group image $G$.

\subsection{Inverse monoids associated to closure operators.}
\begin{Def}[Inverse monoids from closure operators]
	\label{def:inverse-from-closure}
Given an $X$-generated group $G$ with Cayley graph $\Gamma_X$, and a $G$-invariant closure operator $()^c$ on $(P, \subseteq)$ where $P$ is either $\sub \Gamma_X$ or $\csub \Gamma_X$, we can define an inverse monoid $S_c$ where
$$S_c=\{(\Delta, g): g \in G, \Delta \in P \hbox{ is a compact subgraph containing } 1 \hbox{ and } g \}$$
and the operations are
$$(\Delta, g)(\Xi, h)=((\Delta \cup g\Xi)^c, gh)$$
and 
$$(\Delta, g)^{-1}=(g^{-1}\Delta, g^{-1}).$$
\end{Def}

\begin{Lem}
\label{Lem:closure-operator-yields-S}
With the above operations, $S_c$ forms an $E$-unitary inverse monoid with greatest group image $G$. The natural partial order is defined by $(\Delta, g) \leq (\Xi, h)$ whenever $g=h$ and $\Xi \subseteq \Delta$. Furthermore, 
\begin{enumerate}
	\item \label{item:closure-E} if $P=\csub \Gamma_X$, then $S_c$ is $X$-generated in the variety of inverse monoids, and the Sch\"utzenberger graph of $(\Delta, g)$ is isomorphic to $\Delta$. 
	\item \label{item:closure-F} if $P=\sub \Gamma_X$, then $S_c$ is an $F$-inverse monoid with $(\Delta, g)^\m=(\{1,g\}^c, g)$, and is $X$-generated in the variety of $F$-inverse monoids, and the $F$-Sch\"utzenberger graph of $(\Delta, g)$ is equal to $\Delta$.
\end{enumerate}
	
\end{Lem}

\begin{proof}
	First notice that the operations are well-defined as $\Delta \cup g\Xi$ always contains $1$ and $gh$, and it is also connected when $P=\csub \Gamma_X$. The compactness of $(\Delta \cup g\Xi)^c$  follows easily from the compactness of $\Delta$ and $g\Xi$, furthermore $g^{-1}\Delta$ is closed and compact by $G$-invariance, and contains $1$ and $g^{-1}$.
	The multiplication is associative:
	\begin{align*}
		&((\Delta, g)(\Xi, h))(\Omega, k)=((\Delta \cup g\Xi)^c, gh)(\Omega, k)=(((\Delta \cup g\Xi)^c \cup gh\Omega)^c, ghk)=\\
		&=((\Delta \cup g\Xi \cup gh\Omega)^c, ghk), \hbox{ and}\\
		&(\Delta, g)((\Xi, h)(\Omega, k))=(\Delta, g)((\Xi \cup h\Omega)^c, hk)=((\Delta \cup g(\Xi \cup h\Omega)^c)^c, ghk)=\\
		&=((\Delta \cup (g\Xi \cup gh\Omega)^c)^c, ghk)=((\Delta \cup g\Xi \cup gh\Omega)^c, ghk).
	\end{align*}
	It follows that
	\begin{align*}
		&(\Delta, g)(g^{-1}\Delta, g^{-1}) (\Delta, g)=(\Delta \cup g(g^{-1}\Delta) \cup gg^{-1}\Delta, gg^{-1}g)=(\Delta, g), \hbox{ and}\\
		&(g^{-1}\Delta, g^{-1}) (\Delta, g)(g^{-1}\Delta, g^{-1}) =(g^{-1}\Delta \cup g^{-1}\Delta \cup g^{-1}gg^{-1}\Delta, g^{-1}gg^{-1})=(g^{-1}\Delta, g^{-1}).
	\end{align*}
	This shows that $S_c$ is an inverse semigroup, and furthermore notice that $(\{1\}^c, 1)$ is the identity element of $S_c$.
	
	To see that $S_c$ is $E$-unitary with greatest group image $G$, observe that 
	$G$ is a group image of $S_c$ by projection onto the second coordinate, and the preimage of $1$ under this mapping is exactly the set of idempotents
	$$E(S_c) = \{(\Delta,1): \Delta \in P \hbox{ is a compact subgraph containing } 1\}.$$
	This in particular shows that $G$ is the greatest group image, and the $\sigma$-relation is equality in the second component. 
	
	By definition, $(\Delta, g) \leq (\Xi, h)$ holds if and only if $(\Delta, g)=(\Delta, g)(\Delta, g)^{-1}(\Xi, h)=(\Delta, 1)(\Xi,g)=((\Delta \cup \Xi)^c, h)$, which holds exactly when $g=h$ and $\Delta \supseteq \Xi$.
	
	To show (\ref{item:closure-E}), we first show that if $P=\csub \Gamma_X$, then $S_c$ is $X$-generated as an inverse monoid by $\iota: X \to S_c, x \mapsto (\ol x^c,x_G)$. Indeed, take a word $w=x_1^{\epsilon_1} \cdots x_k^{\epsilon_k}$ in $(X \cup X^{-1})^*$, and denote its value in $S_c$, as induced by $\iota$, by $w_{S_c}$. Then it is easy to see that
	$$w_{S_c}=\prod_{i=1}^{k} (\ol x_i^c, (x_i)_G)^{\epsilon_i}=(\langle\ol w\rangle^c, w_G).$$
	Furthermore, if $(\Delta, g) \in S_c$, then $\Delta=F^c$ for some finite connected graph $F$, and we can assume that $1, g \in F$ without loss of generality. Take a word $w$ which labels a path in $\Delta$ from $1$ to $g$, and spans $F$. It then follows that $w_{S_c}=(\Delta, g)$.
	
	Let us determine the Sch\"utzenberger graphs of $S_c$. Since $(\Delta, g)(\Delta, g)^{-1}=(\Delta, 1)$, the $\rrel$-relation is equality in the first component. Thus $$\rrel_{(\Delta, 1)} =V(S\Gamma(\Delta, 1))=\{(\Delta, g): g \in V(\Delta)\}.$$ 
	Furthermore, there is an edge labeled by $x \in X \cup X^{-1}$ from $(\Delta, g)$ to $(\Delta, h)$ in $S\Gamma(\Delta, 1)$ whenever
	$(\Delta,g)(\ol x^c, x_G)=((\Delta \cup g \ol x^c)^c, gx_G)=(\Delta, h)$, which is the case if and only if there is an edge in $\Delta$ from $g$ to $h$ labeled by $x$. This shows that the map $V(S\Gamma(\Delta, 1)) \to \Delta$, $(\Delta, g) \mapsto g$ defines an $X$-labeled graph isomorphism between $S\Gamma(\Delta, 1)$ and $\Delta$.
	
	For (\ref{item:closure-F}), notice that if $P=\sub \Gamma_X$, then $(\{1,g\}^c, g) \sigmarel (\Delta, g)$ for any $\Delta \in P$, moreover, any element of the $\sigma$-class of $(\{1,g\}^c, g)$ is of the form $(\Xi, g)$ with $1,g \in \Xi$, so $(\{1,g\}^c, g) \geq (\Xi, g)$, showing $(\{1,g\}^c, g)=(\Delta, g)^\m$ indeed. 
	
	We claim if $P=\sub \Gamma_X$, then $S_c$ is $X$-generated as an $F$-inverse monoid by the same map $\iota: X \to S_c, x \mapsto (\ol x^c,x_G)$. The proof is analogous to the previous case. For any term $w \in \mathbb I\m_X$, denoting its value in $S_c$ induced by $\iota$ by $w_{S_c}$, it is routine to show that
	$$w_{S_c}=(\langle\ol w\rangle^c, w_G).$$
	It follows that if $(\Delta, g) \in S_c$ with $\Delta=F^c$ for some finite graph $F$, which we can assume contains $1$ and $g$, then we can choose a term $w \in \mathbb I\m_X$ labeling a $1$ to $g$ journey which spans $F$, and thus
$$w_{S_c}=(\Delta,g).$$

	For the statement on $F$-Sch\"utzenberger graphs, let $(\Delta, g) \in S_c$ and consider $F\Gamma(\Delta, g) \subseteq \Gamma_X$ and a vertex $h \in \Gamma_X$. Choose words $u,v \in \mathbb I\m_X$ with $u_G=h$, and $hv_G=g$, and notice that $\ol{u^\m v^\m}$ is a $1$ to $g$ journey in $\Gamma_X$ with $\langle \ol{u^\m v^\m} \rangle=\{1,h,g\}$. So we have $h \in F\Gamma(\Delta, g)$ if and only if $\langle \ol{u^\m v^\m} \rangle \subseteq F\Gamma(\Delta, g)$, which by Proposition \ref{prop:language-of-F-Sch} happens exactly when $(u^\m v^\m)_{S_c}=(\langle \ol{u^\m v^\m} \rangle^{c}, g) \geq (\Delta, g)$. This is in turn true whenever $\langle \ol{u^\m v^\m} \rangle \subseteq \Delta$, which is equivalent to $h \in \Delta$. Thus we have seen that $F\Gamma(\Delta, g)$ and $\Delta$ have the same vertex set.
	
	Proving that they have the same set of edges is similar: taking an edge $e \in \Gamma_X$, we choose  $u,v \in \mathbb I\m_X$ with respect to the vertex $h=\alpha(e)$ as before, and observe that since $\langle \ol{u^\m l(e)l(e)^{-1} v^\m} \rangle = \{1,g\} \cup \langle e \rangle$, $e \in F\Gamma(\Delta, g)$ holds if and only if $\langle \ol{u^\m l(e)l(e)^{-1} v^\m} \rangle \subseteq F\Gamma(\Delta, g)$, which by Proposition \ref{prop:language-of-F-Sch} is equivalent to $\langle \ol{u^\m l(e)l(e)^{-1} v^\m} \rangle \subseteq \Delta$, which in turn is if and only if $e \in \Delta$.
\end{proof}

When we choose $()^c$ to be the identity map, which is indeed a $G$-invariant finitary closure operator, the above construction recovers the Margolis-Meakin expansion of $G=\langle X \rangle$ \cite{MM} in the $E$-unitary case, and its $F$-inverse analogue defined in \cite{AKSz} in the $F$-inverse case. In particular, when $G$ is free on $X$, we obtain the free inverse monoid on $X$ and the free $F$-inverse monoid on $X$, respectively.

The next two statements show that the constructions in Definition \ref{def:inverse-from-closure}, and Definitions \ref{def:E-unit-yields-closure} and \ref{def:F-inverse-yields-closure} are mutually inverse.

\begin{Prop}
	\label{prop:mutually-inverse-E}
		For any $X$-generated, $E$-unitary inverse monoid $S$ with greatest group image $G$, the map $\psi\colon S \to S_{c_S}$, $s \mapsto (S\Gamma(s), s\sigma)$ is a canonical isomorphism. Similarly, for any $X$-generated group $G$ and any $G$-invariant, finitary closure operator $()^c$ on $(\csub \Gamma_X, \subseteq)$, we have $()^c=()^{c_{S_c}}$. 
\end{Prop}

\begin{proof}
	The fact that $\psi$ is an isomorphism is shown in the proof of \cite[Theorem 3.1]{BS03} (but the reader can also adapt the upcoming proof of Proposition \ref{prop:mutually-inverse-F} for a self-contained exposition) and $\psi$ is canonical by definition. 
	
	To show $()^c=()^{c_{S_c}}$, we observe that both closure operators have the same compact sets, which is sufficient as they are both finitary. The compact sets of $()^{c_{S_c}}$ are by Lemma \ref{Lem:S-yields-closure-E-unit} the embedded copies of the Sch\"utzenberger graphs of $S_c$, which by Lemma \ref{Lem:closure-operator-yields-S}, and the definition of $S_c$, and the $G$-invariance of $()^c$, are exactly the compact subgraphs of $()^c$.
\end{proof}

\begin{Prop}
	\label{prop:mutually-inverse-F}
	For any $X$-generated, $F$-inverse monoid $S$ with greatest group image $G$, the map $\psi\colon S \to S_{c_S}$, $s \mapsto (F\Gamma(s), s\sigma)$ is a canonical isomorphism. Similarly, for any $X$-generated group $G$ and any $G$-invariant, finitary closure operator $()^c$ on $(\sub \Gamma_X, \subseteq)$, we have $()^c=()^{c_{S_c}}$. 
\end{Prop}

\begin{Proof}
Note that the map $\psi$ is well-defined, as $F\Gamma(s)$ is compact with respect to $c_S$ and contains $s\sigma$ by definition. To show that $\psi$ it is a morphism (of $F$-inverse monoids), let $s, t \in S$ and choose terms $u, v \in \mathbb I\m _X$ with $u_S=s, v_S=t$, so in particular $\ol u$ is a path from $1$ to $s\sigma$ in $\Gamma_X$. Then we have
\begin{align*}
	&\psi(s)\psi(t)=(\langle \ol u \rangle^{c_S}, s\sigma)(\langle \ol v \rangle^{c_S}, t\sigma)=((\langle \ol u \rangle^{c_S} \cup s\sigma  \langle \ol v \rangle^{c_S})^{c_S}, (st)\sigma)=\\
	&=((\langle \ol u \rangle \cup s\sigma \langle \ol v \rangle)^{c_S}, (st)\sigma)=(\langle \ol{uv} \rangle^{c_S}, (st)\sigma)=\psi(st),
\end{align*}
furthermore 
$$\psi(s)^{-1}=(\langle \ol u \rangle^{c_S}, s\sigma)^{-1}=((s^{-1}\sigma) \langle \ol u \rangle^{c_S}, s^{-1}\sigma)=(\langle \ol {u^{-1}} \rangle^{c_S}, s^{-1}\sigma)=\psi(s^{-1}).$$
Lastly,
$$\psi(s)^\m=(\langle \ol u \rangle^{c_S}, s\sigma)^\m=(\{1, s\sigma\}^{c_S}, s\sigma)=(\langle \ol {u^\m} \rangle^{c_S}, s\sigma)=\psi(s^\m).$$
It is clear by definition that $\psi$ is canonical, which also implies it is surjective.
The fact that $\psi$ is one-to-one follows from Proposition \ref{prop:language-of-F-Sch}: if $u, w \in \mathbb I\m_X$ are such that $\psi(u_s)=\psi(w_s)$, that is, they label coterminal journeys in $\Gamma_X$ from $1$, and  $\langle \ol u \rangle^{c_S}=\langle \ol w \rangle^{c_S}$, then we have both $u_S \geq w_S$ and $w_S \geq u_S$, hence $u_S=w_S$.

The proof of  $()^c=()^{c_{S_c}}$ is analogous to the $E$-unitary case and follows from Lemmas \ref{Lem:S-yields-closure-F-inv} and \ref{Lem:closure-operator-yields-S} similarly.
\end{Proof}

\subsection{McAlister's $P$-theorem}

The framework above allows us to recover Steinberg's proof \cite{BS03} of McAlister's $P$-theorem. In particular given an $X$-generated group $G$ with Cayley graph $\Gamma_X$, and a $G$-invariant, finitary closure operator on $(\csub \Gamma_X, \subseteq)$, the inverse monoid $S_c$ is a $P$-semigroup in the sense of McAlister. 

A $P$-semigroup is an $E$-unitary inverse semigroup associated to a triple $( \mathcal X, \mathcal Y, G)$, where $G$ is a group, $\mathcal Y$ is a meet semilattice, and $\mathcal X$ is a poset containing $\mathcal Y$ as an order ideal. Furthermore, $G$ acts on $\mathcal X$ by order preserving maps, and $\mathcal X$ is the orbit of $\mathcal Y$ under this action. The $P$-semigroup $P(G, \mathcal Y, \mathcal X)$ consist of pairs $(y,g)$ where $g \in G$, and $y \in g\mathcal Y \cap \mathcal Y$, and multiplication is given by $(y,g)(z,h)=(y \wedge gz, gh)$, where $y \wedge gz$ is the greatest lower bound of $y$ and $gz$ in $\mathcal X$ (which the conditions imply exists).

In particular, $S_c$ is the $P$-semigroup associated to the triple $(G, \mathcal Y, \mathcal X)$
where $\mathcal X$ is the set of compact subgraphs of $\csub \Gamma_X$ ordered by reverse inclusion, $\mathcal Y$ is the set of graphs in $\mathcal X$ containing $1$, and the meet in $\mathcal Y$ is given by $\Delta \wedge \Xi=(\Delta \cup \Xi)^c$. Notice that $\mathcal Y$ is indeed an order ideal, and $\mathcal X$ is its orbit under the action of $G$. So Proposition \ref{prop:mutually-inverse-E} gives us McAlister's $P$-theorem for $E$-unitary inverse monoids as an immediate corollary:

\begin{Cor}
Every $E$-unitary inverse monoid is isomorphic to a $P$-semigroup.
\end{Cor}

\subsection{Categories of $X$-generated $E$-unitary and $F$-inverse monoids}

Let $\mathcal E(X,G)$ be the category of $X$-generated $E$-unitary inverse monoids with greatest group image $G$, considered up to canonical isomorphism, where the morphisms are canonical inverse monoid homomorphisms. For a more detailed definition, the reader can consult \cite{MM}, where the category was first introduced. Analogously, define $\mathcal F(X,G)$ to be the category of $X$-generated $F$-inverse monoids with greatest group image $G$, considered up to canonical isomorphism, where the morphisms are canonical $F$-inverse monoid homomorphisms. This category was introduced in \cite{AKSz}.
Notice that given two objects $S,T$ in either category, denoting the natural morphisms onto $G$ from $S$ and $T$ by $\sigma^\sharp_S$ and $\sigma^\sharp_T$ respectively, any canonical morphism $\varphi \colon S \to T$ will satisfy $\sigma^\sharp_T \circ \varphi=\sigma^\sharp_S$.

As usual, we denote the Cayley graph of the $X$-generated group $G$ by $\Gamma_X$. Given such a group $G$, let $\mathcal C(X,G)$ and $\mathcal S(X,G)$ be the category of $G$-invariant, finitary closure operators on $(\csub \Gamma_X, \subseteq)$ and $(\sub \Gamma_X, \subseteq)$ respectively, where there is a morphism from $()^{c}$ to $()^{c'}$ whenever $\Delta^c \subseteq \Delta^{c'}$ for any $\Delta \in \csub \Gamma_X$.  The following equivalent characterization is useful:

\begin{Lem}
\label{lem:closure_morph_equiv}
Let $()^c$, $()^{c'}$ be two closure operators on a poset $(P, \leq)$, with respecitve sets of fixed points $Y_c$ and $Y_{c'}$. Then the following are equivalent:
\begin{enumerate}
\item for any $ x\in P$, we have $x^c \leq x^{c'}$;
\item $Y_c \supseteq Y_{c'}$.
\end{enumerate}
\end{Lem}

\begin{proof}
First assume $Y_c \supseteq Y_{c'}$, and let $x \in P$. Then $x^{c'} \in Y_{c'} \subseteq Y_c$, so $x^{c'}$ is a $c$-closed set above $x$. It follows that $x^c \leq x^{c'}$. Conversely, assume $x^c \leq x^{c'}$ for any $x \in P$ and let $y \in Y_{c'}$. Then $y \leq y^c \leq y^{c'}=y$, so $y=y^c$, which means $y \in Y_c$ indeed.
\end{proof}

The following is a main theorem of the paper.

\begin{Thm}
	\label{Thm:catequivalent}
	\begin{enumerate}
	\item \label{item:catequivalent-E} The categories $\mathcal E(X,G)$ and $\mathcal C(X,G)$ are equivalent.
	\item \label{item:catequivalent-F} The categories $\mathcal F(X,G)$ and $\mathcal S(X,G)$ are equivalent.
	\end{enumerate}
	
\end{Thm}

\begin{proof}
	For (\ref{item:catequivalent-E}), notice that in both categories, there is at most one morphism between any pair of objects. 
	We will show that the maps 
	\begin{align*}
		&\Psi\colon \mathcal E(X,G) \to \mathcal C(X,G), S \mapsto ()^{c_S},\\
		&\Phi \colon  \mathcal C(X,G) \to \mathcal E(X,G), ()^c \to S_c
	\end{align*}
	are functors, that is, there is a morphism $S \to T$ in $\mathcal E(X,G)$ if and only if there is a morphism $()^{c_S} \to ()^{c_T}$ in $\mathcal C(X,G)$, and there is a morphism $()^{c} \to ()^{c'}$ in $\mathcal C(X,G)$ if and only if there is a morphism $S_c \to S_c'$ in $\mathcal E(X,G)$. By Proposition \ref{prop:mutually-inverse-E}, $\Psi \circ \Phi=\id_{\mathcal C(X,G)}$, and as $S$ and $S_{c_S}$ are canonically isomorphic, we also have $\Phi \circ \Psi=\id_{\mathcal E(X,G)}$, therefore the equivalence of the two categories will then follow.
	
	Assume $\varphi\colon S \to T$ is a canonical morphism between $E$-unitary inverse monoids with greatest group image $G$. By Lemma \ref{lem:closure_morph_equiv}, it suffices to show that any $c_T$-closed subgraph of $\Gamma_X$ is also $c_S$-closed.
	Let $\Delta$ be a closed subgraph of $\csub \Gamma_X$ with respect to $c_T$, and let $u \in (X \cup X^{-1})^\ast$ label a path in $\Delta$ from $p$ to $q$, and suppose $u_S=v_S$ for some $v\in (X\cup X^{-1})^*$. Then $u_T=\varphi(u_S)=\varphi(v_S)=v_T$. Since $\Delta$ is $c_T$-closed, $v$ must label a path coterminal with $u$ in $\Delta$, proving that $\Delta$ is $c_S$ closed.
	
	Conversely, assume that there is a morphism $()^{c} \to ()^{c'}$ between finitary, $G$-invariant closure operators on $(\csub, \subseteq)$, that is, we have $\Delta^c \subseteq \Delta^{c'}$ for any $\Delta \in \csub \Gamma_X$. We show that the map $\varphi \colon S_{c} \to S_{c'}, (\Delta, g) \mapsto (\Delta^{c'},g)$ is a morphism. Notice that it is well-defined, as $1, g \in \Delta$ implies $1, g \in \Delta^{c'}$, furthermore it is canonical as 
	$\varphi(\ol x^c, x_G)=(\ol x^{c'}, x_G)$. To see that it is a morphism, observe
	\begin{align*}
		&\varphi((\Delta, g)(\Xi,h))=(((\Delta \cup g\Xi)^c)^{c'},gh)=((\Delta \cup g\Xi)^{c'},gh)\\
		&\varphi(\Delta, g)\varphi(\Xi,h)=(\Delta^{c'}, g)(\Xi^{c'},h)=((\Delta^{c'} \cup g\Xi^{c'})^{c'},gh)=((\Delta \cup g\Xi)^{c'},gh).
	\end{align*}
	This completes the proof of part (\ref{item:catequivalent-E}).
	
	The proof of (\ref{item:catequivalent-F}) is nearly word to word the same, the only difference being that one needs to check the morphism $\varphi \colon S_{c} \to S_{c'}, (\Delta, g) \mapsto (\Delta^{c'},g)$ also respects the operation $^\m$. Indeed it does: $$\varphi((\Delta, g)^\m)=\varphi(\{1,g\}^c, g)=(\{1,g\}^{c'}, g)=(\Delta^{c'},g)^\m.$$
\end{proof}

%\textbf{Note: here we could maybe say more, define the category of not necessarily finitary $G$-invariant closure operators, those also define an $S_c$, and then $\Psi \circ \Phi$ will actually be a closure operator on that category, which is actually just a poset, with fixed points the finitary closure operators. This may help in the next section.}

It follows from Theorem \ref{Thm:catequivalent} that the initial object of $\mathcal E(X,G)$ corresponds to the identical closure operator, that is, it is the Margolis-Meakin expansion of $G$, as proved in \cite{MM}, and the initial object of $\mathcal F(G,X)$ to the $F$-inverse monoid defined in \cite{AKSz}.

\section{Presentations of $F$-inverse monoids}
\label{Sec:presentations}

Since $F$-inverse monoids form a variety, we can talk about presentations of $F$-inverse monoids given in the form of $\FInv \langle X \mid R\rangle$, where $R$ is a relation on $\mathbb I\m_X$. (Often the pairs in $R$ are denoted by equalities.)
The $F$-inverse monoid defined by the presentation is the quotient of the free $F$-inverse monoid on $X$ by the congruence generated by $R$. 
We will also assume that $R$ is symmetric, which we can do without loss of generality by taking the symmetric closure.
This relation satisfies the conditions in Definition \ref{def:F-inverse-yields-closure} with respect to the greatest group image of the monoid defined by the presentation. Therefore any presentation of $F$-inverse monoids gives rise to a closure operator $c_R$ on the Cayley graph of its greatest group image.

The key observation is the following.

\begin{Prop}
\label{prop:F-inverse-repr}
Let $S=\FInv \langle X \mid  R \rangle$ be an $X$-generated $F$-inverse monoid with greatest group image $G$. Then $c_R=c_S$.
\end{Prop}

\begin{Proof}
It is clear by definition that any closed subgraph of $c_S$ is also closed with respect to $c_R$, so there is a morphism $c_R \to c_S$ in the category $\mathcal C(X, G)$, so by Theorem \ref{Thm:catequivalent}, there is also a canonical morphism $S_{c_R} \to S_{c_S} \cong S$. 
Notice furthermore that $S_{c_R}$ is an $X$-generated $F$-inverse monoid satisfying the relations in $R$: indeed if $u=v$ is a relation, then $u_G=v_G$, furthermore $\langle \ol u \rangle \subseteq \langle \ol v \rangle^{c_R}$ and $\langle \ol v \rangle \subseteq \langle \ol u \rangle^{c_R}$, hence 
$(\langle \ol u \rangle^{c_R}, u_G)=(\langle \ol v \rangle^{c_R}, v_G)$ indeed. Therefore there is a canonical morphism $S \to S_{c_R}$, hence $S$ and $S_{c_R}$ are canonically isomorphic. By Theorem  \ref{Thm:catequivalent}, this implies $c_S=c_R$ indeed. 
\end{Proof}

We then obtain the following statement as an immediate consequence of Propositions \ref{prop:FP-expansions} and \ref{prop:F-inverse-repr}.

\begin{Thm}
\label{thm:F-inverse-pres-main}
Let $S=\FInv \langle X \mid  R \rangle$ be an $X$-generated $F$-inverse monoid, and let $\Gamma_X$ be the Cayley graph of its greatest group image. Then the $F$-Schützenberger graph $F\Gamma(w)$ of a term $w \in \mathbb I\m _X$ is $\langle \ol w \rangle^{c_R}$, that is, the smallest $c_R$-closed subgraph of $\Gamma_X$ containing $\ol w$. Furthermore, $F\Gamma(w)=\bigcup_{i=1}^\infty \Delta_i$, where $\Delta_0=\langle \ol w \rangle$ and for $i \geq 1$, $\Delta_{i+1}$ is the full $P$-expansion of $\Delta_i$ with respect to $R$.
\end{Thm}

Recall that by Proposition \ref{prop:mutually-inverse-F}, for two terms $u,v \in \mathbb I\m _X$ we have $u_S=v_S$ if and only if $(F\Gamma(u), u_G)=(F\Gamma(v), v_G)$. In particular, if $u_S=v_S$, and $\Delta_0^u, \Delta_1^u, \ldots$ and $\Delta_0^v, \Delta_1^v, \ldots$ are sequences of finite graphs with $\bigcup_{i=1}^\infty \Delta_i^u=F\Gamma(u)$ and $\bigcup_{i=1}^\infty \Delta_i^v=F\Gamma(v)$, then there exists an index $i$ with $\ol v \subseteq \Delta_i^u$ and $\ol u \subseteq \Delta_i^v$. 
Thus $u_S=v_S$ can be algorithmically verified when true, by iteratively constructing the graphs $\Delta_i^u, \Delta_i^v$ $(i=0,1,...)$ as described in Theorem \ref{thm:F-inverse-pres-main}, and checking whether $\ol v \subseteq \Delta_i^u$ and $\ol u \subseteq \Delta_i^v$ holds.

We close the paper with an example. Consider the free inverse monoid $\FIM(X)$ on $X$. This is well-known to be $F$-inverse, where for $w \in (X \cup X^{-1})^*$, we have $(w_{\FIM(X)})^\m=(\red(w))_{\FIM(X)}$, where $\red(w)$ is the freely reduced form of $w$.

\begin{Prop}
The free inverse monoid $\FIM(X)$ is given by the presentation
$$\FInv \langle X \mid \red(w)=w^\m,\ w \in (X \cup X^{-1})^\ast \rangle.$$
\end{Prop}

\begin{proof}
Let $S=\FInv \langle X \mid \red(w)=w^\m,\ w \in (X \cup X^{-1})^\ast \rangle$. Observe that its greatest group image is the free group on $X$, we denote its Cayley graph by $\Gamma_X$. Consider the closure operator $c_R$ on $\sub\Gamma_X$ corresponding to the relations of the presentation. We claim that the closed subgraphs with respect to $c_R$ are exactly the connected subgraphs of $\Gamma_X$. Indeed, let $\Delta \in \sub \Gamma_X$, and observe that $w^\m$ labels a journey between two vertices $g,h$ of $\Delta$ if and only if $h=g \red(w)$, whereas $\red(w)$ labels a journey between two vertices $g,h$ of $\Delta$ if and only if $h=g \red(w)$ and $\Delta$ contains the unique path from $g$ to $h$. It follows that $\Delta$ is $c_R$-closed if and only if it contains the unique path between any pair of its vertices, i.e. iff it is connected.

It follows that the inverse monoid $S_{c_R}$ consist exactly of pairs $(\Delta, g)$ where $\Delta \in \csub \Gamma_X$ and $1,g \in \Delta$, which is exactly the model of $\FIM(X)$ given by Munn. Since $S \cong S_{c_R}$ by Propositions \ref{prop:mutually-inverse-F} and \ref{prop:F-inverse-repr}, the first claim follows.
\end{proof}

\begin{Prop}
The free inverse monoid $\FIM(X)$ is not finitely presented as an $F$-inverse monoid. 
\end{Prop}

\begin{proof}
Suppose for contradiction that $S=\FInv \langle Y \mid R \rangle$ is a finite presentation for $\FIM(X)$, i.e. $Y$ and $R$ are finite. Let $X=\{x_1, \ldots, x_k\}$ and let $w_i \in \mathbb I\m_Y$ be terms with $(w_i)_S=(x_i)_S$. Then in particular, we have that the set $X$ generates $S$ via the map $\iota_S \colon x_i \mapsto (w_i)_S$, and also generates $S/\sigma=G$ via $\iota_G \colon x_i \mapsto (w_i)_G$.

Let $K$ be the length of the longest relator occurring in $R$, where the length of a term $u_0 v_1^\m u_1  \cdots v_n^\m u_n$ is defined to be $|u_0|+|v_1|+\ldots+|v_n|+|u_n|$. Since $G$ is a free group, there exists a vertex $g \in \Gamma_Y$ such that the distance between $1$ and $g$ in $\Gamma_Y$ is at least $K+1$. We claim that the graph $\{1, g\}$ is $c_R$-closed. Indeed, any term labeling a $1$ to $g$ journey would have length at least $K+1$ by choice, so no relator can be read in the graph $\{1, g\}$. Choose a representative of $g$ in the generators $X$, i.e. a word $w \in (X \cup X^{-1})^\ast$ with $w_G=g$. Then $F\Gamma(w_S^\m) =\{1,g\} \subseteq \Gamma_Y$, whereas  $F\Gamma(\red(w)_S)$ is a subgraph of $\Gamma_Y$ containing a $1$ to $g$ path, hence $w_S^\m \neq \red(w)_S$. But since this contradicts the assumption that $S$ is isomorphic to $\FIM(X)$.
\end{proof}

%\bibliographystyle{abbrv}
%\bibliography{standard2}

\begin{thebibliography}{10}
	
\bibitem{BS} S. Burris, H. P. Sankappanavar, \textit{A Course in Universal Algebra},
Springer New York (1981)



\bibitem{AKSz}  K. Auinger, G. Kudryavtseva, M. B. Szendrei, \textit{F-inverse monoids as algebraic
structures in enriched signature}, Indiana Univ. Math. J. \textbf{70} (2021) 2107–2131.




\bibitem{Law} M.~V.~Lawson, \textit{Inverse semigroups: the Theory of Partial Symmetries}, World Scientific (1998).


 
 \bibitem{MM} S.~W.~Margolis, J.~C.~Meakin, \textit{$E$-unitary inverse monoids and the Cayley graph of a group presentation},
 {J.~Pure Appl.~Algebra}~{\bf 58} (1989) 45--76.
 
 \bibitem{McAl} D. B. McAlister, \textit{Groups, semilattices and inverse semigroups II.} Trans. Amer. Math. Soc. \textbf{196} (1974) 251--270.
 
 
 \bibitem{Ranz99} F. Ranzato, \textit{Closures on CPOs Form Complete Lattices}, Information and Computation,
 \textbf{152} 2 (1999) 236--249.
 

 
 \bibitem{BS03} B. Steinberg, \emph{McAlister's P-Theorem via Schützenberger
 	Graphs}, Communications in Algebra, \textbf{31} 9 (2003) 4387--4392.
 
 
 
 \bibitem{S90} J.~B.~Stephen, \textit{Presentations of inverse monoids}, J. Pure and App. Alg. \textbf{6} (1990) 81--112.
 
 \bibitem{SzSz} N. Szak\'acs, M. B. Szendrei, \textit{On F-inverse covers of finite-above inverse monoids},
 J. Algebra, \textbf{452} (2016) 42–65.

\end{thebibliography}
\end{document}